\documentclass[12pt]{article}
\usepackage{amsmath}
\usepackage{amssymb}
\usepackage{amsmath,amssymb,amsbsy,amsfonts,amsthm,latexsym,
amsopn,amstext,amsxtra,euscript,amscd}

\numberwithin{equation}{section}  
\makeatletter
\renewcommand{\pmod}[1]{\allowbreak\mkern7mu({\operator@font mod}\,\,#1)}
\makeatother
\newcommand{\ssum}[1]{\sum_{\substack{#1}}}  
\renewcommand{\le}{\leqslant}

\renewcommand{\ge}{\geqslant}

\renewcommand{\(}{\left(}
\renewcommand{\)}{\right)}

\newcommand{\pfrac}[2]{\left(\frac{#1}{#2}\right)}  
\newcommand{\eps}{\varepsilon}
\newcommand{\be}{\begin{equation}}
\newcommand{\ee}{\end{equation}}

\begin{document}

\def\A{\mathbb{A}}
\def \C{\mathbb{C}}
\def \F{\mathbb{F}}
\def \K{\mathbb{K}}

\def \Z{\mathbb{Z}}
\def \P{\mathbb{P}}
\def \R{\mathbb{R}}
\def \Q{\mathbb{Q}}
\def \N{\mathbb{N}}
\def \Z{\mathbb{Z}}

\renewcommand{\aa}{\mathbf{a}}

\def\B{\mathcal B}
\def\e{\varepsilon}
\def\a{\alpha}

\def\cA{{\mathcal A}}
\def\cB{{\mathcal B}}
\def\cC{{\mathcal C}}
\def\cD{{\mathcal D}}
\def\cE{{\mathcal E}}
\def\cF{{\mathcal F}}
\def\cG{{\mathcal G}}
\def\cH{{\mathcal H}}
\def\cI{{\mathcal I}}
\def\cJ{{\mathcal J}}
\def\cK{{\mathcal K}}
\def\cL{{\mathcal L}}
\def\cM{{\mathcal M}}
\def\cN{{\mathcal N}}
\def\cO{{\mathcal O}}
\def\cP{{\mathcal P}}
\def\cQ{{\mathcal Q}}
\def\cR{{\mathcal R}}
\def\cS{{\mathcal S}}
\def\cT{{\mathcal T}}
\def\cU{{\mathcal U}}
\def\cV{{\mathcal V}}
\def\cW{{\mathcal W}}
\def\cX{{\mathcal X}}
\def\cY{{\mathcal Y}}
\def\cZ{{\mathcal Z}}

\def\f{\frac{|\A||B|}{|G|}}
\def\AB{|\A\cap B|}
\def \Fq{\F_q}
\def \Fqn{\F_{q^n}}

\def\({\left(}
\def\){\right)}
\def\rf#1{\left\lceil#1\right\rceil}
\def\Res{{\mathrm{Res}}}

\newcommand{\ind}{\operatorname{ind}}

\newcommand{\comm}[1]{\marginpar{
\vskip-\baselineskip \raggedright\footnotesize
\itshape\hrule\smallskip#1\par\smallskip\hrule}}

\newtheorem{lem}{Lemma}
\newtheorem{lemma}[lem]{Lemma}
\newtheorem{prop}{Proposition}
\newtheorem{proposition}[prop]{Proposition }
\newtheorem{thm}{Theorem}
\newtheorem{theorem}[thm]{Theorem}
\newtheorem{cor}{Corollary}
\newtheorem{corollary}[cor]{Corollary}
\newtheorem{prob}{Problem}
\newtheorem{problem}[prob]{Problem}
\newtheorem{ques}{Question}
\newtheorem{question}[ques]{Question}
\newtheorem{rem}{Remark}

\title{On the smallest simultaneous power nonresidue modulo a prime}

\author{
{\sc K. Ford}, {\sc M. Z. Garaev} and {\sc S. V. Konyagin}}

\date{}

\maketitle

\begin{abstract}
Let $p$ be a prime and $p_1,\ldots, p_r$ be distinct prime divisors of
$p-1$. We prove that the smallest positive integer $n$ which is a
simultaneous $p_1,\ldots,p_r$-power nonresidue modulo $p$ satisfies
$$
n<p^{1/4 - c_r+o(1)}\quad(p\to\infty)
$$
for some positive $c_r$ satisfying  $c_r\ge e^{-(1+o(1))r} \; (r\to \infty).$
\end{abstract}

\section{Introduction}

Let $n(p)$ be the
smallest positive quadratic nonresidue modulo $p$ and $g(p)$ be the
smallest positive primitive root modulo $p$. The problem of upper
bound estimates for $n(p)$ and $g(p)$ starts from the early works of
Vinogradov. It is believed that $n(p)= p^{o(1)}$ and $g(p)= p^{o(1)}$
as $p\to\infty$. Vinogradov~\cite{Vin1, Vin4} proved that
$$
n(p)\ll p^{\frac{1}{2\sqrt{e}}}(\log p)^2,\qquad g(p)< \frac{2^{k+1}(p-1)p^{\frac{1}{2}}}{\phi(p-1)},
$$
where $k$ is the number of distinct prime divisors of $p-1$. Hua~\cite{Hua} improved Vinogradov's result to
$g(p)<2^{k+1}p^{1/2}$ and then Erd\H
os and Shapiro~\cite{ES} refined it to $g(p)\ll
k^Cp^{\frac{1}{2}}$, where $C$ is an absolute constant. These bounds
were improved by Burgess~\cite{Bur1, Bur2} to
$$
n(p)<p^{\frac{1}{4\sqrt{e}} +o(1)},\qquad g(p)< p^{\frac{1}{4}+o(1)} \qquad (p\to\infty).
$$
The Burgess bounds remains essentially the best known up to  date,
in a sense that it is not even known that $n(p)\ll p^{1/4\sqrt{e}}$
or that $g(p)\ll p^{1/4}$.

If one allows a small exceptional set of primes, then better estimates may be obtained.
Using his ``large sieve'', Linnik \cite{Li} proved that for any $\eps>0$, there are only
$O_\eps(\log\log x)$ primes $p\le x$ for which $n(p)>p^\eps$.  The sharpest to date results for
$g(p)$ (which also hold for the least \emph{prime} primitive root modulo $p$) are due to Martin
\cite{Ma}, who proved that for any $\eps>0$, there is a $C>0$ so that $g(p)=O( (\log p)^C )$ with
at most $O(x^\eps)$ exceptions $p\le x$.  All of these type of results are ``purely existential'', in that
one cannot say for which specific primes $p$ the bounds hold (say, in terms of the factorization of $p-1$).

From elementary considerations it follows that an integer $g$ is
a primitive root modulo $p$ if and only if for any prime divisor $q
| p-1$ the number $g$ is a $q$-th power nonresidue modulo $p$. Thus,
if $p_1,\ldots,p_k$ are all the distinct prime divisors of $p-1$,
then $g(p)$ is the smallest positive simultaneous $p_1,\ldots,p_k$-th
power nonresidue modulo $p$. In the present paper we prove the
following result.

\begin{theorem}
\label{thm:main} Let $p$ be a prime number and $p_1,\ldots, p_r$ be
distinct prime divisors of $p-1$.  Then the smallest positive integer
$n$ which is a simultaneous $p_1,\ldots,p_r$-th power nonresidue
modulo $p$ satisfies
$$
n<p^{1/4-c_r}e^{C(\log r)^{1/2}(\log p)^{1/2}}
$$
where $C>0$ is an absolute constant and  $c_r\ge e^{-(1+o(1))r}$ as $r\to \infty.$
\end{theorem}

The novelty of the result is given by the factor $p^{-c_r}$. We
observe that for
$c_r<(\log p)^{-1/2}$ (in particular, for
$r\ge(0.5+\e)\log\log p$
and $p\ge p(\e)$) this factor
is  dominated by the exponential factor.

The following corollaries directly follow from
Theorem~\ref{thm:main}.

\begin{corollary}
\label{cor:main1} Let $p$ be a prime number and $p_1,\ldots, p_r$ be
distinct prime divisors of $p-1$, where $r$ is fixed. Then the
smallest positive integer $n$ which is a simultaneous
$p_1,\ldots,p_r$-th power nonresidue modulo $p$ satisfies
$$
n<p^{1/4 - c_r+o(1)}\quad(p\to\infty).
$$
\end{corollary}

From our earlier discussion, the upper bound given in Theorem \ref{thm:main} holds also for
$g(p)$ whenever $p-1$ has $r$ distinct prime factors.

\begin{corollary}
\label{cor:main2} For any $\e>0$, if $p-1$ has at most $(0.5
-\e)\log\log p$ distinct prime divisors, then $g(p) =
o(p^{1/4})$ as $p\to\infty$.
\end{corollary}

The counting function of primes satisfying the hypothesis of Corollary \ref{cor:main2} is
$x(\log x)^{-3/2+(\log 2)/2-O(\eps)}$ (the upper bound follows from e.g., \cite[Inequality (5)]{Erd35};
the lower bound can be obtained using sieve methods).

\begin{rem}\label{remark1}
 The focus of our arguments is to establish bounds which are uniform in $r$.
 We have made no attempt to optimize the value of $c_r$ for small $r$, and leave
 this as a problem for further study.
\end{rem}

Our proof of Theorem \ref{thm:main} proceeds in three main steps.  The first is a standard application
of character sums to show that a large proportion of integers $n<p^{1/4+o(1)}$  are simultaneous
$p_1,\ldots,p_r$-th power nonresidue modulo $p$.  Next, we show that if such a number $n$ has many divisors
($r2^r$ divisors suffice), then for some pair $d<d'$ of these divisors, the smaller number
$n'=dn/d'$ is also a   simultaneous  $p_1,\ldots,p_r$-th power nonresidue modulo $p$.  This procedure is
most efficient when the ratios $d'/d$ are uniformly large.  In the third step we show that
integers possessing many well-spaced divisors are sufficiently dense, so that there must be one such number
in the set guaranteed by first step (with an appropriate quantification of ``well-spaced'' and ``dense'').

\section{Character sums and distribution of power nonresidues}

We begin by recalling the well-known character sum estimate of Burgess
\cite{Bur2, Bur}.

\begin{lem}
\label{lem:Burg} If $p$ is a prime and $\chi$ is a non-principal
character modulo $p$ and if $H$ and $m$ are arbitrary positive
integers, then
$$
\Bigl|\sum_{n=N+1}^{N+H} \chi(n)\Bigr| \ll H^{1- 1/m} p^{(m+1)/4m^2}(\log p)^{1/m}
$$
for any integer $N$, where the implied constant is absolute.
\end{lem}

See the proof in~\cite{IK}, (12.58). In the remark after the proof
the authors announce that the factor $(\log p)^{1/m}$ can be
replaced by $(\log p)^{1/(2m)}$, but this is not important for us.

\begin{lem}
\label{lem:simul} Let $p$ be a prime number and $p_1,\ldots,p_r$ be
distinct prime divisors of $p-1$. The number $J$ of integers
$n\le H$ which are simultaneous $p_1,\ldots,p_r$-th power
nonresidues modulo $p$ satisfies
$$
J\ge 0.12H\prod_{i=1}^r\Bigl(1-\frac{1}{p_i}\Bigr) + O\Bigl(r^{13} H^{1- 1/m} p^{(m+1)/4m^2}(\log p)^{1/m}\Bigr),
$$
where the constant implied in the $``O"$-symbol is
absolute.
\end{lem}

\begin{proof}
 We follow the method of~\cite{E}.  Let $C$ be a sufficiently large constant, to be chosen later.
Assuming that
$p_1<\dots<p_r$, we choose the largest $s\le
r$ so that $p_s\le Cr^2$ (if $p_1>Cr^2$, then set $s=0$). Let $J_1$ be the number of
integers $n\le H$ which are simultaneous $p_1,\ldots,p_s$-th power
nonresidues modulo $p$. For $j>s$, let $J_{2,j}$ be the number of
integers $n\le H$ which are $p_j$-th power residues modulo $p$.
Clearly,
\begin{equation}
\label{Jlow}
J\ge J_1 - \sum_{j=s+1}^r J_{2,j}.
\end{equation}

Let $g$ be a primitive root of $p$ and let
$\chi_0$ be the principal Dirichlet character modulo $p$.  We will denote by $\chi$
a generic Dirichlet character modulo $p$.
By orthogonality, for $(x,p)=1$ we have
$$
\frac{1}{d}\sum_{\chi^d=\chi_0}\chi(x) = \begin{cases} 1, & \text{ if }\, \ind_gx\equiv 0\pmod d,\\
0, & \text{ if } \,\ind_gx\not\equiv 0\pmod d.\end{cases}
$$
A number $n$ is a $p_i$-power residue
modulo $p$ if and only if $p_i| \ind_g n$.
Hence,
$$
J_1=\sum_{\substack{n\le H\\\gcd(\ind_g n, p_1\ldots p_s)=1}}1
=\sum_{d|p_1\ldots p_s}\mu(d)\sum_{\substack{n\le H\\ d|\ind_g n}} 1
$$
and for $j=s+1,\dots,r$ we have
\begin{equation}
\label{eqn:J2j=}
J_{2,j}= \sum_{\substack{n\le H \\ p_j | \ind_g n}} 1.
\end{equation}

We denote
$$R=H^{1- 1/m} p^{(m+1)/4m^2}(\log p)^{1/m}.$$
Using Lemma~\ref{lem:Burg} for
$\chi\not=\chi_0$, we get for any $d$ that
\begin{equation}\label{sumchid}
\sum_{\substack{n\le H \\ d|\ind_g n}} 1 =
  \frac 1d\sum_{\chi^d=\chi_0}\sum_{n\le H}\chi(n) = \frac{H}{d} + O(R).
\end{equation}

To estimate $J_1$ we use a lower bound sieve as in~\cite{E}
combining with \eqref{sumchid}.  Brun's sieve \cite[Theorem 2.1 and the following Remark 2]{HR}
suffices.  Here the ``sieve dimension'' is $\kappa=1$.  Taking $\lambda=\frac14$, $b=1$, $z=Cr^2$
and $L=O(R)$ in \cite[Theorem 2.1 and the following Remark 2]{HR}, we get that
\begin{align*}
 J_1 &\ge H \prod_{i=1}^s \(1-\frac{1}{p_i}\) \(1 - 2 \frac{\lambda^{2b} e^{2\lambda}}{1-\lambda^2 e^{2+2\lambda}}
+O\pfrac{1}{\log z} \) - O(z^{4.1} R)\\
&\ge  0.13 H \prod_{i=1}^s\(1-\frac{1}{p_i}\) - O(r^{13} R)
 \end{align*}
if $C$ is large enough.

By \eqref{eqn:J2j=} and \eqref{sumchid},
$$\sum_{j=s+1}^r J_{2,j} = H\sum_{j=s+1}^r\frac{1}{p_j} +O(rR) \le \frac{H}{Cr}  +O(rR),$$
since $p_j>Cr^2$ for all $j\ge s+1$.
Invoking \eqref{rprod} and assuming that $C\ge 100$, we get
$$J_1 - \sum_{j=s+1}^r J_{2,j} \ge 0.12 H\prod_{i=1}^s\Bigl(1-\frac{1}{p_i}\Bigr)
+ O(r^{13}R).$$
Using (\ref{Jlow}) we complete the proof of the lemma.
\end{proof}

\section{Reduction of simultaneous nonresidues}

The aim of this section is to show that if a positive integer $n$
which is a simultaneous $p_1,\ldots,p_r$-th power nonresidue modulo
$p$ has many divisors then it is possible to construct $n'<n$ which
is also a simultaneous $p_1,\ldots,p_r$-th power nonresidue modulo
$p$.

\begin{lem}
\label{lem:NotEqual} Let $a$ be a non-zero real number, $\ell\in \mathbb{N}$ and
\begin{equation}
\label{eqn:seq1}
a_1,a_2,\ldots,a_{2\ell-1}
\end{equation}
be any sequence of $2\ell-1$ real numbers (not necessarily distinct).
Then for some indices $i_1<i_2<\ldots <i_\ell$ we have that
$a_{i_{s}}-a_{i_{t}}\not = a$ for all $1\le s,t\le \ell$.
\end{lem}

\begin{proof}
We may assume that $a>0$.  Define an equivalence relation on the numbers $i$ by
setting $i \sim j$ if $a_i-a_j=ka$ for some integer $k$.  Let $S_1,\ldots,S_m$ be the
different (nonempty) equivalence classes.  Clearly $a_i-a_j=a$ is only possible for $i,j$ within
a given equivalence class.  Let $b_r$ be the smallest element of $S_r$, for each $r=1,\ldots,m$.
Divide each $S_r$ into two subclasses,
\begin{align*}
 S_r^{(0)} &= \{ i \in S_r : a_i-a_{b_r} = ka \text{ for some even integer } k\}, \\
 S_r^{(1)} &= \{ i \in S_r : a_i-a_{b_r} = ka \text{ for some odd integer } k\}.
\end{align*}
Obviously $a_i-a_j=a$ is impossible within each subclass $S_r^{(0)}, S_r^{(1)}$.
For $1\le r\le m$, define $\eps_r=0$ if $|S_r^{(0)}| \ge |S_r^{(1)}|$, and $\eps_r=1$ otherwise, and
let $B=\bigcup_{r=1}^m S_r^{(\eps_r)}$.  Then $|B| \ge \ell$, and $a_i-a_j\ne a$ for $i,j\in B$.
Any set $\{i_1,\ldots,i_\ell\} \subset B$ then satisfies the requirements of the lemma.
\end{proof}

\begin{rem}\label{lem1_remark}
 The conclusion of Lemma \ref{lem:NotEqual} is best possible, as may be seen by taking $a_i=ai$ for
 $1\le i\le 2\ell-1$; in any set of $\ell+1$ elements $a_i$ there are two with difference $a$.
\end{rem}

\begin{lem}
\label{lem:NotCongr} Let $q$ be a prime, $u\in\R,\, u>1$ and
$a\in\Z$, $a\not\equiv 0\pmod q$. Assume that
\begin{equation}
\label{eqn:seq2}
a_1,a_2,\ldots,a_{t}
\end{equation}
is a sequence of $t\ge 2uq/(q-1)$ integers (not necessarily
distinct). Then for some $\ell\in\N,\,\ell \ge u$ and indices
$i_1<i_2<\ldots <i_\ell$ we have that
$$
a_{i_{v}}-a_{i_{w}}\not \equiv a \pmod q \quad (1\le v,w\le \ell).
$$
\end{lem}

\begin{proof} We can assume that $a=1$. Define $\ell=\lceil u\rceil$.
From the pigeon-hole principle, there is a residue class $h\pmod q$
containing at most $t/q$ elements from the sequence~\eqref{eqn:seq2}. Since
$$
\Big\lceil t - \frac{t}{q}\Big\rceil = \Big\lceil t(q-1)/q\Big\rceil
\ge \lceil 2u \rceil\ge 2\ell-1,
$$
after rearranging \eqref{eqn:seq2} we may assume that
$$
a_{s}\not\equiv h\pmod q \quad (s=1,2,\ldots, 2\ell-1).
$$
Define $c_{s}\in \{1,2,\ldots, q-1\}$ by
$$
c_{s}\equiv a_{s}-h\pmod q.
$$
By Lemma~\ref{lem:NotEqual}, there is a subsequence
$c_{i_1},\ldots,c_{i_{\ell}}$ such that
$$
c_{i_v}-c_{i_w}\not =1 \quad (1\le v,w\le \ell).
$$
Since $1\le c_i\le q-1$ this implies that
$$
c_{i_v}-c_{i_w}\not \equiv 1\pmod q \quad (1\le v,w\le \ell)
$$
and thus
\[
a_{i_v}-a_{i_w}\not \equiv 1\pmod q \quad (1\le v,w\le \ell). \qedhere
\]
\end{proof}

\begin{rem} \label{rem1}
For $q=2$ it is enough to require $t\ge 2u$.
Indeed, we can choose a large subsequence of
$a_1,a_2,\ldots,a_{t}$ of the same parity.
\end{rem}

\begin{cor}
\label{cor:cor2} Let $p_1,p_2,\ldots,p_r$ be prime numbers, and
$$
\mathbf{b}=(b_1,b_2,\ldots,b_r)\in \F_{p_1}^{*}\times\F_{p_2}^{*}\times \ldots \times \F_{p_r}^{*}.
$$
Let
$$
t> 2^r\prod_{i:p_i>2}\frac{p_i}{p_i-1}
$$
and
$$
\aa_1,\, \aa_2\,\ldots, \aa_t
$$
be a sequence of $t$ elements from $\F_{p_1}\times\F_{p_2}\times
\ldots \times \F_{p_r}$. Then for some $i<j$ we have that
$$
\aa_{j}-\aa_{i}\in \left(\F_{p_1}\setminus \{b_1\}\right)\times
\left(\F_{p_2}\setminus \{b_2\}\right)\times
\ldots \times \left(\F_{p_r}\setminus \{b_r\}\right).
$$
\end{cor}

Corollary~\ref{cor:cor2} follows from $r$ applications of
Lemma~\ref{lem:NotCongr} and taking into account Remark~\ref{rem1}.

\begin{cor}
\label{cor:cor3} Let $p$ be a prime number and suppose $p_1,\ldots, p_r$ are
distinct prime divisors of $p-1$. Let $n$ be a simultaneous
$p_1,\ldots,p_r$-th power nonresidue modulo $p$ and $d_1<\dots<d_t$
be some divisors of $n$ where
$$
t> 2^r\prod_{p_i>2}\frac{p_i}{p_i-1}.
$$
Then there exists $i,j$ such that $1\le i<j\le t$ and the number
$n'=nd_i/d_j$ is also a simultaneous $p_1,\ldots,p_r$-th power
nonresidue modulo $p$.
\end{cor}

\begin{proof} Let $g$ be a primitive root modulo $p$. To each number $x$ we
associate the vector
$$
(u_1,u_2,\ldots,u_r)\in \F_{p_1}\times\F_{p_2}\times \ldots \times \F_{p_r},
$$
so that for $1\le i\le r$, $x\equiv g^{p_i k_i+s_i} \pmod{p}$ where $0\le s_i<p_i$

Let the vector $(b_1,b_2,\ldots, b_r)$ correspond to $n$ and the
vectors $\aa_1, \aa_2, \ldots, \aa_t$
correspond to $d_1,\dots,d_t$, respectively.  Apply Corollary~\ref{cor:cor2} and select the
indices $i<j$ such that
$$
\aa_j-\aa_i\in \left(\F_{p_1}\setminus \{b_1\}\right)\times\left(\F_{p_2}\setminus \{b_2\}\right)
\times \ldots \times \left(\F_{p_r}\setminus \{b_r\}\right)
$$
Then $n'=nd_i/d_j$ is a simultaneous $p_1,p_2,\ldots, p_r$-power  nonresidue
modulo $p$.
\end{proof}

\begin{rem}\label{rem:cor3}
We note that if $p_1,p_2,\ldots, p_r$ are distinct primes, then
\be\label{rprod}
r> \prod_{p_i>2}\frac{p_i}{p_i-1}.
\ee
Hence, in Corollaries \ref{cor:cor2} and \ref{cor:cor3} one can take
$t=2^r r$.
\end{rem}

\section{Integers with well-spaced divisors}

Let $P^{-}(n)$ and $P^+(n)$ denote the smallest and
largest prime factor of $n$, respectively, let $\omega(n)$ be the number
of distinct prime factors of $n$, and let $\tau(n)$ be the number of positive
divisors of $n$.

\bigskip

\begin{lem} \label{lem:separ}
For each fixed constant $c>1/\log2=1.442\ldots$, there is
$\eta=\eta(c)>0$ such that the following holds.
Uniformly for integers $t$, $2\le t\le (\log x)^{1/c}$, all but
$O_{c}(x/t^\eta)$ integers $n\le x$ have $t$ divisors $d_1<d_2<\cdots<d_t$
such that $d_{j+1}/d_j > x^{1/t^{c}}$ for all $1\le j\le t-1$.
\end{lem}

\begin{proof}
We may assume that $t\ge 10$.  Take $$\e = \frac{c-1/\log 2}{4}, \qquad
\alpha=1/\log 2+\e.$$
Write each $n\le x$ in the form $abd$ where
$P^-(d)>x^{1/\log t}$, $P^+(a)\le x^{1/(t^\a\log t)}$ and all prime
factors of $b$ lie in $(x^{1/(t^\a\log t)},x^{1/\log t}]$.
We divide $n$ into several categories.  Let $k_0= \lceil \frac{\log
  2t}{\log 2} \rceil$.  Let $S_0$ be the set of
$n\le x$ with either $d=1$ or with $b$ not squarefree.  Let $S_1$ be the
set of $n$ with $d>1$, $b$ squarefree and $\omega(b) < k_0$.
We denote $\a_j=j\e$ for $1\le j\le J-1:=[\a/\e]$, $\a_J=\a$,
$a_j=x^{1/(t^{\a_j}\log t)}$ for $j=1,\dots,J$.
Let $S_2$ be the set of $n$ with $d>1$, $b$ squarefree and
the number of primes from the interval $(a_j,x^{1/\log t}]$ dividing $n$
is less than $k_j:=(\a_j-\e)\log t$ for some $j=1,\dots,J-1$.
Let $S_3$ be the set of the remaining integers $n$.

We first show that $S_0$, $S_1$, and $S_2$ are small.  By standard counts for
smooth numbers,
\[
|S_0| \le \Psi(x,x^{1/\log t})+\sum_{p>x^{1/(t^\a\log t)}}
\frac{x}{p^2} \ll \frac{x}{t}+\frac{x}{x^{1/(t^\a\log t)}} \ll \frac{x}{t}.
\]
Next, by the results of Hal\'asz~\cite{Hal} on the number of integers with a prescribed number of prime
factors from a given set (see also Theorem 08 of \cite{HT}), we have
\begin{align*}
|S_1| &\ll \sum_{k < k_0} x e^{-E} \frac{E^k}{k!}, \qquad
E=\sum_{x^{1/(t^\a\log t)}<p\le x^{1/\log t}} \frac{1}{p}=\a\log t + O(1) \\
&\ll x t^{-\a} \sum_{k < k_0} \frac{(\a \log t)^k}{k!} \\
&\ll x (t^\a)^{-(\beta\log\beta-\beta+1)}, \quad \beta=\frac{1}{\a \log 2}=\frac1{1+\e \log 2}<1
\\
&\ll x / t^{\delta}
\end{align*}
for some $\delta>0$ which depends on $\e$.

For any $j=1,\dots,J-1$ we denote by $S_{2,j}$ the set of $n\le x$ with
less than $k_j$ prime divisors from $(a_j,x^{1/\log t}]$. We have
$$
|S_{2,j}|\ll \sum_{k < k_j} x e^{-E_j} \frac{E_j^k}{k!},
$$
where
$$E_j=\sum_{x^{1/(t^{\a_j}\log t)}<p\le x^{1/\log t}} \frac{1}{p} = \alpha_j\log t + O(1).$$
Arguing as before we get
$$|S_{2,j}|\ll x / t^{\delta'}$$
for some $\delta'>0$ which depends on $\e$.

Notice that for $n\in S_3$, $\tau(b) =2^{\omega(b)} \ge 2^{k_0} \ge 2t$.
Next, let $S_4$ be the set of $n\in S_3$ for which $b$ does \emph{not} have
$t$ well-spaced divisors in the sense of the lemma.  Since $d>1$ for
such $n$, given such a \emph{bad} value of $b$, using a standard sieve
bound the number of choices
for the pair $(a,d)$ is bounded above by
\[
\sum_a |\{d\le x/ab:P^-(d)>x^{1/\log t}\}| \ll \sum_a
\frac{x/ab}{\log(x^{1/\log t})} \ll \frac{x}{bt^\a}.
\]
Hence,
\begin{equation}\label{S4est}
|S_4|\ll  \sum_{\text{bad }b}\frac{x}{bt^\a}
\end{equation}

A number $b$ which is bad has many pairs of \emph{neighbor
  divisors}.  To be
precise, let $\sigma=t^{-c}\log x$ and define
\[
W^*(b;\sigma) = |\{(d',d''):d'|b,d''|b,d'\ne d'',|\log (d'/d'')| \le \sigma \}|.
\]
Let $d_1<\dots<d_{\tau(b)}$ be the divisors of $b$. We construct
the subsequence $D_1<\dots<D_r$ of this sequence:
$$D_1=1,\quad D_i=\min\{d_j:\,d_j> x^{t^{-c}}D_{i-1}\}\,(i>1).$$
The process is terminated if $D_i$ does not exist. Let
$D_{r+1}=+\infty$. The set $\{d_1,\dots,d_{\tau(b)}\}$ is divided
into $r$ subsets $\mathcal D_i$, $i=1,\dots,r$, where
$$\mathcal D_i=\{d_j:\,D_i\le d_j<D_{i+1}\}.$$
We see that $(d',d'')$ is counted in $W^*(b;\sigma)$ if $d',d''\in\mathcal D_i$
for some $i$ and $d'\neq d''$. Hence,
$$W^*(b;\sigma)\ge\sum_{i=1}^r|\mathcal D_i|(|\mathcal D_i|-1)
=\sum_{i=1}^r |\mathcal D_i|^2 - \tau(b).$$
Since $\tau(b)\ge 2t$ and $r\le t$, we get by the Cauchy-Schwartz inequality that
\[
\tau(b)^2 = \bigg(\sum_{i=1}^r |\mathcal D_i| \bigg)^2 \le t \bigg( \sum_{i=1}^r |\mathcal D_i|^2 \bigg) \le t (W^*(b;\sigma)+\tau(b))
\le tW^*(b;\sigma)+\frac12 \tau(b)^2.
\]
Therefore,
\begin{equation}\label{Sum1over_b}
\sum_{\text{bad }b} \frac{1}{b} \le \sum_{\text{all }b}
\frac{2W^*(b;\sigma)t}{b\tau(b)^2},
\end{equation}
each sum being over squarefree integers whose prime factors lie in
$(x^{1/(t^\a\log t)},x^{1/\log t}]$.

In the latter sum, fix $k=\omega(b)$, write $b=p_1\cdots p_k$, where
the $p_i$ are primes, and $p_1<\cdots<p_k$.
Then $W^*(p_1\cdots p_k;\sigma)$ counts the number of pairs
$Y,Z \subset \{1,\ldots,k\}$ with $Y\ne Z$ and
\begin{equation}\label{YZpi}
\Big| \sum_{i\in Y} \log p_i - \sum_{i\in Z} \log p_i \Big| \le \sigma.
\end{equation}
Fix $Y, Z$, and let $I$ be the maximum element of the symmetric difference
$(Y\cup Z)-(Y\cap Z)$.  We fix $I$ and count the number of $p_1,\ldots,p_k$ satisfying \eqref{YZpi}.
We further partition the solutions, according to the condition
$a_j<p_I\le a_{j-1}$, for $j=1,\ldots,J$.  Fix the value of $j$.
If all the $p_i$ are fixed except for $p_I$, then
\eqref{YZpi} implies that $p_I$ lies in some interval of the form
$[U,Ue^{2\sigma}]$.  As $p_I>x^{1/t^{\a_j}\log t}$ as well, and $\a>c$,
we have (putting $U_j=\max(U,x^{1/t^{\a_j}\log t})$)
\[
\sum_{p_I} \frac{1}{p_I} \ll \log\(1+\frac{2\sigma}{\log U_j}\) \ll
\frac{\sigma}{\log U_j} \ll t^{\a_j-c}\log t.
\]
Hence, for each fixed $k$, $j$, $Y$ and $Z$,
\be\label{sump1pk}
\begin{split}
 \ssum{x^{1/t^{\a}\log t} < p_1<\ldots<p_k\le x^{1/\log t}}
   \frac{1}{p_1\cdots p_k}
  &\ll \frac{t^{\a_j-c} (\log t)}{(k-1)!}
\Big(\sum_{x^{1/t^{\a}\log t} < p\le x^{1/\log t}} \frac{1}{p}
\Big)^{k-1} \\ &\ll \frac{t^{\a_j-c}(\log t)(\a\log t+ O(1))^{k-1}}{(k-1)!}.
\end{split}
\ee
Now we estimate the number $N(I,j)$ of choices for the pair $Y,Z$ for fixed $I$
and $j$.  Since $p_I \le a_{j-1}$, the condition $n\in S_3$ implies
$I\le k-k_{j-1}$.
For any $i\le I$ there
are at most four possibilities: $i\in Y\cap Z$, $i\in Y\setminus Z$,
$i\in Z\setminus Y$, $i\not\in Y\cup Z$. For $i>I$
there are two possibilities: $i\in Y\cap Z$ and
$i\not\in Y\cup Z$. Therefore,
\be\label{NIj}
N(I,j)\le 4^{I}2^{k-I} \le 4^k 2^{-k_{j-1}} \le 4^k t^{-\a_j\log 2+2\e\log 2}.
\ee
It follows from \eqref{sump1pk} and \eqref{NIj} that
\[
\sum_{\omega(b)=k} \frac{W^*(b;\sigma)t}{b\tau(b)^2}
\ll \sum_{j=1}^J t^{1+(1-\log2)\a_j+2\e-c}
\sum_k \frac{(\a\log t+ O(1))^{k-1}}{(k-1)!}.
\]

Taking into account that $\a_j\le\a$ and summing on $j,k$ we get
\[
\sum_{b} \frac{W^*(b;\sigma)t}{b\tau(b)^2} \ll t^{1+2\e+(2-\log 2)\a-c}.
\]
Thus, by (\ref{S4est}) and (\ref{Sum1over_b}),
\[
|S_4| \ll \frac{x}{t^{c-(1-\log 2)\a-2\e-1}} = \frac{x}{t^{c-1/\log 2 - \e(3-\log 2)}} \ll \frac{x}{t^\e}.
\]
 Therefore, there are
$x-O(x/t^{\min(\delta,\delta',\e)})$ numbers $n\le x$
for which $b$ does have $t$ well-spaced divisors.
\end{proof}

\begin{rem}\label{rem_div}
 Lemma \ref{lem:separ} is best possible in the sense that the conclusion does not hold for
$c<1/\log 2$.
 In fact, for any $c<1/\log 2$,
 the number of integers $n\le x$ that \emph{do} have $t$ divisors $d_1,\ldots,d_t$ with
 $d_{j+1}/d_j < n^{1/t^c}$ for all $j$ is $O_c(x/t^\eta)$ for some $\eta>0$ which depends on $c$.

\begin{proof}
 It is well-known
 that if $t$ is large, $c<1/\log 2$ and $\eps$ small enough, then a typical integer
 $n$ has $r\sim (c+\e)\log t$ prime factors $p_1,\ldots,p_r$ in $[n^{1/t^{c+\e}},n]$.
 This can be seen, e.g. by the theorem of Hal\'asz used in the estimation of $|S_1|$.
 In fact, the  number of exceptional  $n\le x$ is $O_c(x/t^\eta)$.
 Thus,  a typical $n$ has
 about $2^{(c+\e)\log t}=t^{(c+\e)\log 2}<t$ divisors composed of such primes.
Also, for most of these $n$, $n/(p_1 \cdots p_r) < n^{1/(2t^c)}$; by Theorem 07 of \cite{HT}, the
number of exceptions $n\le x$ is $O(x \exp \{-c_1 t^{\eps} \})$ for some $c_1>0$.
Suppose that such an $n$ has $t$ well-spaced divisors $d_1,\ldots,d_t$ with
 $d_{j+1}/d_j < n^{1/t^c}$ for all $j$.  By the pigeon-hole principle, two of these divisors share
 the same set of prime factors from $\{p_1,\ldots,p_r\}$, hence their ratio is less than $n^{1/(2t^c)}$, a
 contradiction.
 \end{proof}
 \end{rem}

\section{Proof of Theorem~\ref{thm:main}}

We rewrite the assertion of Lemma~\ref{lem:simul} as
\begin{equation}
\label{est2J}
J\ge 0.12H\prod_{i=1}^r\Bigl(1-\frac{1}{p_i}\Bigr) - R',\quad
R'=(5r)^{C''} H^{1- 1/m} p^{(m+1)/4m^2}(\log p)^{1/m}
\end{equation}
for some constant $C''$. Let $\mathcal{N}$ denotes the set of $n\in[1,H]$ which are simultaneous
$p_1,\ldots,p_r$-th power nonresidue modulo $p$, where
$$
H=p^{1/4}e^{(C''+3)(\log p)^{1/2} (\log(5r))^{1/2}}\log p.
$$
Assume that $p$ is sufficiently large, and take
$$m=\lfloor (\log p)^{1/2} (\log(5r))^{-1/2}\rfloor.$$
Notice that $m\gg(\log p)^{1/2}(\log\log p)^{-1/2}\to\infty$ as
$p\to\infty$.  Since
$$R' = H (Hp^{-1/4}/\log p)^{-1/m} p^{1/(4m^2)}(5r)^{C''},$$
we have
$$(Hp^{-1/4}/\log p)^{-1/m}\le(5r)^{-C''-3}$$
and
$$p^{1/(4m^2)}\le 5r.$$
Consequently,
$$R'\le H(5r)^{-2}.$$
By \eqref{est2J} and \eqref{rprod},
$$J\ge (0.12r^{-1} - (5r)^{-2})H \ge 0.08H/r.$$
So, we see that
\begin{equation}
\label{estmathcalN}
|\mathcal{N}|\ge 0.08H/r.
\end{equation}

We consider the case
\begin{equation}
\label{firstcase}
r<0.6\log\log p
\end{equation}
first.  We will apply Lemma~\ref{lem:separ} with $x=H$,
fixed $c\in(1/\log2,1.5]$, and with $t=K r 2^r$, where
$K$ is a sufficiently large constant depending on $c$.
By \eqref{estmathcalN}, the exceptional set in
Lemma \ref{lem:separ} is smaller than $|\mathcal{N}|$ provided that  $K$ is large enough.
The condition $2\le t\le (\log x)^{1/c}$
is satisfied due to the restriction on $r$ and $c$.
By Lemma \ref{lem:separ}, for some $n\in\mathcal{N}$, there are well-separated
divisors $d_1<\dots<d_t$ of $n$, satisfying $d_{i+1}/d_i>n^{1/t^c}$ for each $i$.
Now we are in position to apply
Corollary~\ref{cor:cor3} and we see that there is an $n'\le np^{-t^{-c}/4}$ such
that $n'$ is a simultaneous $p_1,\ldots,p_r$-th power nonresidue
modulo $p$.  Noting that
$t^{-c}/4=\exp\{-r(c\log 2+o(1))\}$ and that $c$ may be taken arbitrarily close to $1/\log 2$,
we complete the proof.

If~(\ref{firstcase}) does not hold, then, as
we have mentioned in Section~1, the factor $p^{-c_r}$ in the
statement of the theorem is dominated by the second factor, and the claim follows
from the fact that $\mathcal{N}\neq\emptyset$.

\section{Acknowledgements}
The first author is supported in part by National Science Foundation grants
DMS-1201442 and DMS-1501982.
The third author is supported by grant RFBR 14-01-00332
and grant Leading Scientific Schools N 3082.2014.1.

Address of the authors:\\

K. Ford, Department of Mathematics, 1409 West Green Street,
University of Illinois at Urbana-Champaign, Urbana, IL 61801, USA.

E-mail address: {\tt ford@math.uiuc.edu}

\vspace{1cm}

M.~Z.~Garaev, Centro de Ciencias Matem\'{a}ticas, Universidad
Nacional Aut\'onoma de M\'{e}xico, C.P. 58089, Morelia,
Michoac\'{a}n, M\'{e}xico.

Email address: {\tt garaev@matmor.unam.mx}

\vspace{1cm}

S. V. Konyagin, Steklov Mathematical Institute, 8 Gubkin Street,
Moscow, 119991, Russia.

Email address: {\tt konyagin@mi.ras.ru}


\begin{thebibliography}{100}

\bibitem{Bur1} D. A.  Burgess,
`The distribution of quadratic residues and non-residues', {\it
Mathematika}, {\bf  4} (1957), 106--112.

\bibitem{Bur2} D. A. Burgess, `On character sums and primitive roots',
{\it Proc. London Math. Soc.}, {\bf 12} (1962), 179--192.

\bibitem{Bur} D. A. Burgess, `On character sums and $L$-series. II.',
{\it Proc. London Math. Soc.}, {\bf 13} (1963), 524--536.

\bibitem{Erd35}  P. Erd\H os,  `On the normal number of prime factors of $p-1$
and some related problems concerning Euler's $\varphi$-function',
{\it Quart.\ J.\ Math.\ Oxford Ser.}, {\bf6} (1935), 205--213.

\bibitem{E} P. Erd\H os, `On the least primitive root of a prime',
{\it Bull. London Math. Soc.}, {\bf 55} (1945), 131--132.

\bibitem{ES} P. Erd\H os, H. N. Shapiro, `On the least primitive root
of a prime',  {\it Pacific J. Math.,} {\bf 7} (1957), 861--865.


\bibitem{Hal} G. Hal\'asz,  `Remarks to my paper: ``On the distribution of additive and the mean values of multiplicative arithmetic functions''',
{\it Acta Math. Acad. Sci. Hungar.}, {\bf 23} (1972), 425--432.

\bibitem{HR} H.~Halberstam, H.-E.~Richert, {\it Sieve methods}, Academic Press, 1974.

\bibitem{Hua} L.-K. Hua,  `On the least primitive root of a prime',
{\it Bull. Amer. Math. Soc.,} {\bf 48} (1942), 726--730.

\bibitem{HT}
R. R. Hall, G. Tenenbaum, {\em Divisors,} Cambridge
Tracts in mathematics vol. {\bf 90}, 1988.

\bibitem{IK} H. Iwaniec, E. Kowalski, {\it  Analytic number theory},
American Mathematical Society, Providence, Rhode Island, 2004.


\bibitem{Li}  Yu. V. Linnik, `A remark on the least quadratic non-residue', (Russian)
{\it C. R. (Doklady) Acad. Sci. URSS (N.S.)}, {\bf 36} (1942), 119--120.


\bibitem{Ma} G. Martin,
`The least prime primitive root and the shifted sieve',
{\it Acta Arith.}, {\bf 80} (1997), no. 3, 277--288.


\bibitem{Vin1} I. M. Vinogradov, `On the distribution of quadratic residues and nonresidues',
(in Russian), {\it Journal of the Physico-Mathematical Society of Perm}, 1919.


\bibitem{Vin4} I. M. Vinogradov, `On the least primitive root', (Russian), {\it Doklady Akad. Nauk  SSSR}, {\bf 1} (1930),  7--11.


\end{thebibliography}
\end{document}